\documentclass[10pt]{amsart}
\usepackage{amsmath, amsthm, amscd, amsfonts,graphicx}
\usepackage{amsmath,amsthm,amsthm, amscd, amsfonts, stackrel, latexsym, dsfont,amssymb,mathrsfs,textcomp,wasysym,hyperref}
\usepackage{tikz}
\usepackage{tikz-cd}
\usetikzlibrary{matrix,arrows}
\usepackage[all]{xy}

\usepackage{color} % Required for custom colors
\definecolor{slateblue}{rgb}{0,0.0,0.8}

\usepackage{xcolor}
\usepackage{graphicx}
\hypersetup{%
	colorlinks=true,% hyperlinks will be coloured
	linkcolor=slateblue,% hyperlink text will be green
	linkbordercolor=gray,% hyperlink border will be red
	citecolor=slateblue
}

\theoremstyle{definition}
\newtheorem{definition}{Definition}[section]
\newtheorem{example}[definition]{Example}
\theoremstyle{remark}
\newtheorem{remark}[definition]{Remark}
\newtheorem{question}[definition]{Question} 
\theoremstyle{plain}
\newtheorem{theorem}[definition]{Theorem}
\newtheorem{lemma}[definition]{Lemma}
\newtheorem{proposition}[definition]{Proposition}
\newtheorem{corollary}[definition]{Corollary}

\usepackage{mathptmx}

\makeatletter
\@namedef{subjclassname@2020}{\textup{2020} Mathematics Subject Classification}
\makeatother

\usepackage{amsthm}

\newtheorem{innercustomgeneric}{\customgenericname}
\providecommand{\customgenericname}{}
\newcommand{\newcustomtheorem}[2]{%
	\newenvironment{#1}[1]
	{%
		\renewcommand\customgenericname{#2}%
		\renewcommand\theinnercustomgeneric{##1}%
		\innercustomgeneric
	}
	{\endinnercustomgeneric}
}

\newcustomtheorem{customthm}{Theorem}
\newcustomtheorem{customlemma}{Lemma}

\newenvironment{myproof}[2] {\textsc{Proof of Theorem }}{\hfill$\square$}

\numberwithin{equation}{section}

\begin{document}
	
	\title{A remark on stability and the D-topology of mapping spaces}
	
	\author{Alireza Ahmadi}
	
	\address{Department of Mathematical Sciences, Yazd University, Yazd, 89195--741, Iran}
	
	\email{ahmadi@stu.yazd.ac.ir; alirezaahmadi13@yahoo.com}

	\subjclass[2020]{Primary 58A05, 58C25; Secondary 57P99}
	
	\keywords{Mapping spaces, D-topology, stability theorem, diffeological \'{e}tale manifolds}
	
	\begin{abstract}
		We discuss how stability is related to the D-topology of mapping spaces, equipped with the functional diffeology. Indeed, we show that stable classes of mapping spaces are D-open. After a reformulation of the classical stability theorem of manifolds with respect to the D-topology, we prove a version of the stability theorem in the class of diffeological \'{e}tale manifolds.
	\end{abstract}
	
	\maketitle
	
	\section{Introduction}
	The stability of mapping spaces of manifolds\footnote{By a manifold we mean  a second-countable
		Hausdorff  finite-dimensional smooth manifold.}  is an aspect of the classical differential geometry
	which concerns those properties that remain invariant under small deformations or perturbations in a smooth manner (see, e.g., \cite{GG,GP}). Let us first recall  the classical definition of a stable property.
	\begin{definition}\label{def-0}(\cite{GP})
		Let $ M $ and $ N $ be manifolds.
		A property $ \mathsf{P} $ of smooth maps is \textbf{stable} on $ \mathrm{C}^{\infty}(M,N) $  under small deformations, whenever for any smooth map $ h:M\times[0,1]\rightarrow N $, where $ [0,1] $ is a manifold with boundary, if $ h_0:M\rightarrow N, x\mapsto h(x,0) $ possesses the property $ \mathsf{P} $, then there exists a number $ \epsilon>0 $ such that the map $ h_t:M\rightarrow N, x\mapsto h(x,t) $  also possesses the property $ \mathsf{P} $, for all $ t\in[0,\epsilon) $.
	\end{definition}
	
	This article is aimed to study the stability of mapping spaces in the context of diffeology, which extends ordinary differential geometry by diffeological spaces. The reader can refer to the book \cite{PIZ} for
	a comprehensive introduction to diffeology. 
	One advantage of working in this framework, among others, is that mapping spaces have a natural structure, called the functional diffeology, for which the natural map
	\begin{center}
		$ \mathrm{C}^{\infty}(X,\mathrm{C}^{\infty}(Y,Z)) \longrightarrow\mathrm{C}^{\infty}(X\times Y,Z) $
	\end{center}
	taking $ f \mapsto \tilde{f}$ with $\tilde{f}(x,y)=f(x)(y) $, is a diffeomorphism of diffeological spaces 
	(see \cite[\S 1.60]{PIZ}). 
	This feature turns the category of diffeological spaces into a Cartesian closed one, which is very useful in smooth homotopy theory. This particularly enables us to define a coherent notion of stability in diffeology.
	\begin{definition}\label{def-1}
		Let $ \mathsf{P} $ be a property for smooth maps between diffeological spaces $ X $ and $ Y $, i.e., $ \mathrm{C}^{\infty}(X,Y) $. 
		We say that the property $ \mathsf{P} $ is \textbf{stable} on $ \mathrm{C}^{\infty}(X,Y) $  under small deformations whenever for any smooth homotopy $ h:\mathbb{R} \rightarrow \mathrm{C}^{\infty}(X,Y) $, if  $ h(0) $ possesses the property $ \mathsf{P} $, then there exists a number $ \epsilon>0 $ such that  $ h(t) $ also possesses the property $ \mathsf{P} $, for all $ t\in(-\epsilon,\epsilon) $.
	\end{definition}
	
	On the other hand, every diffeological space has an intrinsic topology, called the D-topology, in which a subset is D-open\footnote{Throughout this article, the prefix D of a topological property indicates the same property in terms of the D-topology.} if its preimage by any plot is open.
	Hence a mapping space obtain a topological structure from the functional diffeology. 
	\begin{proposition}\label{p-3} 
		Let $ X $ and $ Y $  be diffeological spaces.
		A property $ \mathsf{P} $ is stable on $ \mathrm{C}^{\infty}(X,Y) $ if and only if 
		\begin{center}
			$ \mathsf{class}(\mathsf{P})=\{ f\in\mathrm{C}^{\infty}(X,Y)\mid f $ possesses the property $ \mathsf{P} \} $
		\end{center}
		is a D-open subset of $ \mathrm{C}^{\infty}(X,Y) $, where $ \mathrm{C}^{\infty}(X,Y) $ is equipped with the functional diffeology.
	\end{proposition} 
	\begin{proof}
		As the D-topology of a diffeological space is determined via smooth paths by \cite[Theorem 3.7]{CSW2014}, the proof is straightforward.
	\end{proof}  
	
	\begin{corollary}
		It is immediate that for any stable property $ \mathsf{P} $,
		\begin{center}
			$ \mathsf{dim}(\mathsf{class}(\mathsf{P}))\leq \mathsf{dim}(\mathrm{C}^{\infty}(X,Y)) $,
		\end{center}
		where $ \mathsf{dim} $ denotes the diffeological dimension \cite[\S 1.78]{PIZ}.
	\end{corollary}   
	
	To compare the classical definition of stability with the diffeological one in the class of manifolds, we need the following lemma.
	But before that, consider the space $ [0,1] $ equipped with the diffeology of manifolds with boundary (see \cite[\S 4.16]{PIZ} and \cite{GI}).
	The D-topology of $ [0,1] $ is the subspace topology inherited from $ \mathbb{R} $.
	\begin{lemma}\label{lem-5}
		Let  $ X $ be a diffeological space.
		A subset $ A $ of $ X $ is D-open if and only if for every $ h\in \mathrm{C}^{\infty}([0,1],X) $  
		with $ h(0)\in A $, there exists a number $ \epsilon>0 $ such that  $ h(t)\in A $, for all $ t\in[0,\epsilon) $.
	\end{lemma}
	
	\begin{proof}
		The proof is inspired by that of \cite[Theorem 3.7]{CSW2014}:
		Since smooth maps are D-continuous, if a subset $ A $ of $ X $ is D-open and
		$ h\in \mathrm{C}^{\infty}([0,1],X) $  with $ h(0)\in A $, there exists a number $ \epsilon>0 $ such that $ h(t)\in A $, for all $ t\in[0,\epsilon) $.
		
		Suppose that $ A \subseteq X $ and for every $ h\in \mathrm{C}^{\infty}([0,1],X) $  
		with $ h(0)\in A $, there exists a number $ \epsilon>0 $ such that  $ h(t)\in A $, for all $ t\in[0,\epsilon) $. To see that $ A $ is D-open in $ X $, take any plot $ P:U\rightarrow X $. Let $ r\in P^{-1}(A) $ and $ \{r_n\} $ be a sequence which converges fast to $ r $ in $ U $, up to suitable subsequences. By the special curve lemma \cite[p. 18]{KM},
		there is a smooth map $ C:\mathbb{R}\rightarrow U $ with 
		$ C(\frac{1}{n})=r_n $ and $ C(0)=r $. Set $ c=C|_{[0,1]} $. Then $ P\circ c:[0,1]\rightarrow X $ is an element of $ \mathrm{C}^{\infty}([0,1],X) $ with $ P\circ c(0)\in A $. 
		By hypothesis, there exists a number $ \epsilon>0 $ such that  $ P\circ c(t)\in A $, for all $ t\in[0,\epsilon) $. 
		Thus, for sufficiently large $ n $ we get,
		$ P\circ c(\frac{1}{n})\in A $ and consequently,
		$ r_n=c(\frac{1}{n})\in P^{-1}(A) $. Therefore, $ P^{-1}(A) \subseteq U $ is open.
	\end{proof}  
	
	\begin{corollary}\label{cor-2}
		Let $ M $ and $ N $ be manifolds.
		A property $ \mathsf{P} $ is stable on $ \mathrm{C}^{\infty}(M,N) $ (according to Definition \ref{def-0}) if and only if
		\begin{center}
			$ \mathsf{class}(\mathsf{P})=\{ f\in\mathrm{C}^{\infty}(M,N)\mid f $ possesses the property $ \mathsf{P} \} $
		\end{center}
		is a D-open subset of $ \mathrm{C}^{\infty}(M,N) $, where $ \mathrm{C}^{\infty}(M,N) $ is equipped with the functional diffeology.
	\end{corollary} 
	
	\begin{proposition} 
		In the class of manifolds, Definitions \ref{def-0} and \ref{def-1} are equivalent, meaning that a property $ \mathsf{P} $  is  stable according to Definition \ref{def-0} if and only if it is stable according to Definition \ref{def-1}.
	\end{proposition}

	\begin{proof}
		The proof is a consequence of Proposition \ref{p-3} and Corollary \ref{cor-2}.
	\end{proof} 
	\subsection{Stability theorem}
	
	By Corollary \ref{cor-2}, the classical stability theorem of manifolds (see, e.g., \cite[p. 35]{GP})  can  be rephrased with respect to the D-topology in the following way.
	
	\begin{theorem}\label{the-stb}(Stability theorem of manifolds).
		The following classes of smooth maps from a compact manifold $ K $ to a manifold $ M $ are D-open subsets of $ \mathrm{C}^{\infty}(K,M) $, where $ \mathrm{C}^{\infty}(K,M) $ is equipped with the functional diffeology.
		\begin{enumerate}
			\item[(i)] 
			diffeomorphisms,
			\item[(ii)] 
			\'{e}tale maps (i.e., local diffeomorphisms),
			\item[(iii)] 
			submersions,
			\item[(iv)] 
			immersions,
			\item[(v)] 
			embeddings,
			\item[(vi)] 
			maps transversal to any specified closed submanifold $ N\subseteq M $.
		\end{enumerate}
	\end{theorem}  
	
	\begin{remark}
		Besides the D-topology, there are other well-known topologies on mapping spaces of manifolds such as the compact-open
		topology, the weak topology, the strong topology, and the Whitney topology. 
		In \cite[Theorem 43.1]{KM} or \cite[Section 5]{Mic}, the stability theorem is shown with respect to the Whitney topology. 
		Furthermore, in \cite[Chapter 2]{Hir}, results similar  to the stability theorem is proved with respect to the strong topology.
		As mentioned in \cite[Corollary 4.15]{CSW2014}, one can reach Theorem \ref{the-stb} by comparisons made in \cite{CSW2014}, between the D-topology and the strong topology, and also other topologies. Of course, it seems that way would be longer and more complicated than our approach. 
	\end{remark}  
	As an application, this restatement of the stability theorem can be helpful in some computations. 
	\begin{example} 
		Let  $ K $ be a compact manifold.
		In light of 
		\cite[Propositions 3.6]{CW}, also
		\cite[p. 72]{Hec} and \cite[Propositions 6.3]{HM-V}, or equivalently \cite[Corollary 4.29]{CW},
		one can compute the internal tangent spaces of the following subspaces of $ \mathrm{C}^{\infty}(K,K) $
		at $ \mathrm{id}_K $, which
		are all isomorphic to the vector space of all smooth vector fields on $ K $:
		\begin{enumerate}
			\item 
			diffeomorphisms,
			\item
			\'{e}tale maps (equivalently, submersions or immersions),
			\item  
			embeddings,
			\item 
			maps transversal to any specified closed submanifold $ N\subseteq K $.
		\end{enumerate}
	\end{example}

	It is natural to suggest such a problem for mapping spaces of diffeological spaces.
	\begin{question} (Stability problem).
		Let $ X $ and $ Y $ be  diffeological spaces. Which of the following classes and under what conditions are D-open in $ \mathrm{C}^{\infty}(X,Y) $?	
		\begin{enumerate}
			\item[$ \bullet $] 
			diffeomorphisms,
			\item[$ \bullet $] 
			(local) subductions,
			\item[$ \bullet $] 
			(local) inductions,
			\item[$ \bullet $] 
			(diffeological) submersions,
			\item[$ \bullet $] 
			(diffeological) immersions,
			\item[$ \bullet $]  
			(diffeological) \'{e}tale maps,
			\item[$ \bullet $] 
			(diffeological) embeddings,
			\item[$ \bullet $]
			etc.
		\end{enumerate}
		(See Section \ref{S2} for terminology). 
	\end{question}
	Although the stability problem remains open in general, thanks to some linear-algebraic tools %such as Lemma \ref{p-1} 
	and additional topological conditions, we are able to generalize the stability theorem to diffeological \'{e}tale manifolds, which is the main result of this article. 
	Diffeological \'{e}tale manifolds, introduced by the author in \cite{ARA}, constitute a class of diffeological spaces  that includes the usual manifolds and  also irrational tori. We
	briefly review this kind of diffeological space in Section \ref{S2}. 
	
	\begin{theorem}\label{thm-main}
		(Stability theorem of diffeological \'{e}tale manifolds).
		Suppose that $ \mathcal{K} $ and $ \mathcal{M} $ are diffeological \'{e}tale manifolds such that $ \mathcal{K} $ is D-compact (i.e., compact with respect to the D-topology). Also, $ \mathrm{C}^{\infty}(\mathcal{K},\mathcal{M}) $ is equipped with the functional diffeology.
		\begin{enumerate}
			\item[$ \textbf{(a)} $] 
			The following classes of $ \mathrm{C}^{\infty}(\mathcal{K},\mathcal{M}) $ are D-open:
			\begin{enumerate}
				\item[$ \textbf{(a1)} $] 
				diffeological submersions,
				\item[$ \textbf{(a2)} $] 
				diffeological immersions,
				\item[$ \textbf{(a3)} $] 
				diffeological \'{e}tale maps,
				\item[$ \textbf{(a4)} $] 
				local subductions (provided that $ \mathcal{M} $ is D-Hausdorff).
			\end{enumerate}
			\item[$ \textbf{(b)} $] 
			If $ \mathcal{M} $ is a usual manifold, then
			the following classes of $ \mathrm{C}^{\infty}(\mathcal{K},\mathcal{M}) $ are D-open:
			\begin{enumerate}
				\item[$ \textbf{(b1)} $] injective diffeological immersions,
				\item[$ \textbf{(b2)} $] diffeological embeddings,
				\item[$ \textbf{(b3)} $] diffeomorphisms.
			\end{enumerate}
			
			\item[$ \textbf{(c)} $] 
			Let $ \mathsf{Inj}^{\infty}(\mathcal{K},\mathcal{M})\subseteq\mathrm{C}^{\infty}(\mathcal{K},\mathcal{M}) $ be the diffeological subspace of injective smooth maps from $ \mathcal{K} $ to $ \mathcal{M} $.
			If $ \mathcal{M} $ is D-Hausdorff,  then
			the following classes of $ \mathsf{Inj}^{\infty}(\mathcal{K},\mathcal{M}) $ are D-open subsets of $ \mathsf{Inj}^{\infty}(\mathcal{K},\mathcal{M}) $.
			\begin{enumerate}
				\item[$ \textbf{(c1)} $] diffeological embeddings,
				\item[$ \textbf{(c2)} $]  diffeomorphisms.
			\end{enumerate}
		\end{enumerate}	
	\end{theorem}
	\begin{remark}
		Notice that 
		we have to consider requirements that guarantee the injectivity remains stable in the cases of diffeological embeddings and diffeomorphisms. We can take two ways to handle this difficulty: 1) to consider the restrictive condition that $ \mathcal{M} $ to be a usual manifold in item $ \textbf{(b)} $ so that diffeological immersions are locally injective (see Proposition \ref{p-2}) and follow the classical method of proof, or 2) to impose directly injectivity condition in $ \textbf{(c)} $.
		We  take into account diffeomorphisms just as a subset in $ \textbf{(b)} $ and $ \textbf{(c)} $, without any refined diffeological structure.
	\end{remark}  
	\section{Diffeological \'{e}tale manifolds}\label{S2}
	In this section, we briefly recall the needed definitions and results about diffeological \'{e}tale manifolds (see \cite{ARA} for more details).
	
	\begin{definition}
		We call a smooth map $ f:X\rightarrow Y $ between diffeological spaces a \textbf{submersion} if for each $ x_0 $ in $ X $, there exists a smooth local section $ \sigma:O\rightarrow X $ of $ f $
		passing through $ x_0 $ defined on a D-open subset $ O\subseteq Y $ such that $ f\circ\sigma(y)=y $ for all $ y\in O $.
		A smooth map $ f:X\rightarrow Y $  is said to be a \textbf{diffeological submersion} if
		the pullback  of $ f $ by every plot in $ Y $ is a submersion.
	\end{definition}
	
	\begin{proposition} 
		If $ f:X\rightarrow M $ is a diffeological submersion into a manifold $ M $, then it is a submersion.
		In particular, 	a map between manifolds is a diffeological submersion if and only if it is a submersion of manifolds.
	\end{proposition}
	
	\begin{definition} 
		A smooth map $ f:X\rightarrow Y $ between diffeological spaces is an \textbf{immersion} if for each $ x_0 $ in $ X $, there exist a D-open neighborhood $ O\subseteq X $ of the point $ x_0 $, a D-open neighborhood $ O'\subseteq Y $ of the set $ f(O) $, and a smooth map $ \varrho:O'\rightarrow X $ such that 
		$ \varrho\circ f(x)=x $ for all $ x\in O $. 
		A smooth map $ f:X\rightarrow Y $  is a \textbf{diffeological immersion} if
		for any plot $ P:U\rightarrow Y $ in $ Y $,
		for each $ (r_0,x_0) $ in $ P^*X $, there exist a D-open neighborhood $ O $ of $ (r_0,x_0) $ in $ P^*X $, an open neighborhood $ V\subseteq U $ of $ P^*f(O) $ and a smooth map $ \varrho:V\rightarrow U\times X $ such that $ \varrho\circ P^*f(r,x)=(r,x) $ for all $ (r,x)\in O $.
		\begin{displaymath}
			\xymatrix{
				U\times X\ar@/^0.4cm/[drr]^{\Pr_2}& & \\
				&	P^*X \ar[ul]\ar[r]^{P_{\#}}\ar[d]^{P^*f} & X\ar[d]^{f} \\
				&	V\ar@/^0.6cm/@{-->}[uul]^{\varrho}\ar[r] \subseteq U \ar[r]^{P} & Y  }
		\end{displaymath}
	\end{definition}
	
	Any immersion of diffeological spaces is a local induction.
	\begin{proposition}\label{p-2} 
		If $ f:X\rightarrow M $ is a diffeological immersion into a manifold $ M $, then it is an immersion.
		In particular, 	a map between manifolds is a diffeological immersion if and only if it is an immersion of manifolds.
	\end{proposition}
	
	\begin{definition}
		A map $ f:X\rightarrow Y $ between diffeological spaces is called a \textbf{diffeological embedding} if it is both a diffeological immersion and a D-embedding (i.e., a topological embedding with respect to the D-topology).
	\end{definition} 
	Any diffeological embedding  is an induction.
	
	\begin{definition}
		A map $ f:X\rightarrow Y $ between diffeological spaces is \textbf{\'{e}tale} if for every $ x $ in $ X $, there are D-open neighborhoods $ O\subseteq X $ and $ V\subseteq Y $ of $ x $ and $ f(x) $, respectively, such that $ f|_O:O\rightarrow O' $ is a diffeomorphism.  
		A smooth map $ f:X\rightarrow Y $ is a \textbf{diffeological \'{e}tale map} if the pullback $ P^{*}f $ by every plot $ P $ in $ X $ is  \'{e}tale. 
	\end{definition}
	
	\begin{proposition} 
		If $ f:X\rightarrow M $ is a diffeological \'{e}tale map into a manifold $ M $, then it is \'{e}tale.
		In particular, 	a map between manifolds is a diffeological \'{e}tale map if and only if it is a local diffeomorphism of manifolds.
	\end{proposition}
	
	\begin{remark}
		Any diffeological \'{e}tale map and more generally, any diffeological submersion is a D-open map. 
	\end{remark}
	
	\begin{definition} 
		A diffeological space $ \mathcal{M} $ is said to be a \textbf{diffeological \'{e}tale $ n $-manifold} if there exists a parametrized cover $ \mathfrak{A} $, called an \textbf{atlas} for $ \mathcal{M} $ consisting of diffeological \'{e}tale maps 
		from $ n $-domains into $ \mathcal{M} $.
		We call the elements of $ \mathfrak{A} $ \textbf{diffeological \'{e}tale charts}.
	\end{definition} 
	
	In this situation, $ \mathfrak{A} $ is actually a covering generating family for $ \mathcal{M} $ and the diffeological dimension of $ \mathcal{M} $ is equal to $ n $, $ \mathsf{dim}(\mathcal{M})=n $. Moreover, at each point $ x $ of $ \mathcal{M} $, the internal tangent space $ T_{x}\mathcal{M} $ is isomorphic to $ \mathbb{R}^n $.
	Therefore, for a diffeological \'{e}tale $ n $-manifold, the diffeological dimension and the dimension of the tangent spaces are the same.
	
	If $ \varphi:U\rightarrow \mathcal{M} $ are $ \psi:V\rightarrow \mathcal{M} $ are two diffeological \'{e}tale charts in $ \mathcal{M} $ with $ \varphi(r)=\psi(s) $, then we can find a unique smooth map $ h:U'\rightarrow V $ defined on an open neighborhood $ U'\subseteq U $ of $ r $ such that $ h(r)=s $ and the following diagram commutes:
	\begin{displaymath}
		\xymatrix{
			&  \mathcal{M}  & \\
			U' \ar[ur]^{\varphi|_{U'}}\ar[rr]_{h} & & V\ar[ul]_{\psi}  \\
		}
	\end{displaymath}
	Although diffeological \'{e}tale charts  $ \varphi $ are $ \psi $ may not locally injective, $ h $ is indeed an \'{e}tale map and uniquely determined by $ \varphi $ are $ \psi $. Locally, $ h $ plays the role of a transition map.
	
	\begin{example}\label{exa-1}
		Diffeological manifolds are diffeological \'{e}tale manifolds but not conversely.
		The irrational torus $ \mathbb{T}_{\alpha}=\mathbb{R}/(\mathbb{Z}+\alpha\mathbb{Z}) $,  $ \alpha\notin\mathbb{Q}  $,  is a D-compact diffeological \'{e}tale $ 1 $-manifold which is not a diffeological manifold.
		More generally,	$ \mathbb{T}^n_{\Gamma}=\mathbb{R}^n/\Gamma $, where $ \Gamma $ is a discrete subgroup of $ \mathbb{R}^n $, is a diffeological \'{e}tale $ n $-manifold. 
	\end{example}
	
	\begin{example}
		Any open subset of a diffeological \'{e}tale $ n $-manifold itself is  a diffeological \'{e}tale $ n $-manifold
		in a natural way. 
	\end{example}
	
	\begin{example}(Product \'{e}tale manifolds).
		Suppose $ \mathcal{M}_1,\dots,\mathcal{M}_k $ are diffeological \'{e}tale manifolds of dimensions $ n_1,\dots,n_k $, respectively. The product space $ \mathcal{M}_1\times\cdots\times\mathcal{M}_k $
		is a diffeological \'{e}tale manifold of dimension $ n_1+\cdots+n_k $ with
		diffeological \'{e}tale charts of the form $ \varphi_1\times\cdots\times\varphi_k $, where $ \varphi_i $ is a diffeological \'{e}tale chart of  $ \mathcal{M}_i $, $ i=1,\dots,k $.
	\end{example}
	
	Recall that the internal tangent map of a smooth map is the linear map between the internal tangent spaces induced by universal property (see \cite{CW,Hec,HM-V}).
	
	\begin{theorem}\label{the-2} 
		Let $ \mathcal{M} $ and $ \mathcal{N} $ be diffeological \'{e}tale manifolds and
		$ f:\mathcal{M}\rightarrow \mathcal{N} $ be a smooth map.
		\begin{enumerate}
			\item[(i)]
			$ f $ is a diffeological submersion if and only if the internal tangent map $ df_x:T_x\mathcal{M}\rightarrow T_{f(x)}\mathcal{N} $ is an epimorphism at each point $ x\in\mathcal{M} $.
			\item[(ii)] 
			$ f $ is a diffeological immersion
			if and only if the internal tangent map $ df_x:T_x\mathcal{M}\rightarrow T_{f(x)}\mathcal{N} $ is a monomorphism at each point $ x\in\mathcal{M} $.
			\item[(iii)] 
			$ f $ is a diffeological \'{e}tale map if and only if the internal tangent map $ df_x:T_x\mathcal{M}\rightarrow T_{f(x)}\mathcal{N} $ is an isomorphism at each point $ x\in\mathcal{M} $.
		\end{enumerate}
		
	\end{theorem}

	\begin{definition}
		A smooth map $ f:\mathcal{M}\rightarrow \mathcal{N} $ between  diffeological \'{e}tale manifolds has \textbf{internal
			rank} $ k $ at a point $ x\in \mathcal{M} $ if the rank of the internal tangent map $ df_x:T_x\mathcal{M}\rightarrow T_{f(x)}\mathcal{N} $ is equal to $ k $. Moreover,
		$ f $	has \textbf{full rank} at $ x $ if its internal
		rank is equal to $ \min\{\mathsf{dim}(\mathcal{M}),\mathsf{dim}(\mathcal{N})\} $.
	\end{definition}
	
	\begin{proposition}\label{cor-1}
		Let $ f:\mathcal{M}\rightarrow \mathcal{N} $ be a smooth map and $ x_0\in \mathcal{M} $.
		If  $ f $ has full rank $ k $ at $ x_0 $,
		then there is a D-open neighborhood $ O\subseteq \mathcal{M} $ of $ x_0 $ such that
		$ f $ has full rank $ k $ on $ O $.
	\end{proposition} 	
	\begin{proof}
		Take any diffeological \'{e}tale chart $ \varphi:U\rightarrow \mathcal{M} $ with $ \varphi(r_0)=x_0 $  for some $ r_0\in U $.
		The composition $ f\circ\varphi $ is a plot in $ \mathcal{N} $. So there are
		diffeological \'{e}tale chart $ \psi:V\rightarrow \mathcal{N} $ and a smooth map $ F:U'\rightarrow V $ defined on an open neighborhood of $ r_0\in U'\subseteq U $ such that $ f\circ\varphi|_{U'}=\psi\circ F $.
		Then
		$ df_{\varphi(r_0)}\circ d\varphi_{r_0}=d\psi_{F(r_0)}\circ dF_{r_0} $.
		Since $ d\varphi_{r_0} $ and $d\psi_{F(r_0)} $ are isomorphisms, $ F $ has full rank $ k $ at $ r_0 $.
		By \cite[Proposition 4.1]{Lee-SM}, $ r_0 $ has a neighborhood $ U''\subseteq U' $ such that $ F $ has full rank $ k $ on $ U'' $.
		Therefore, $ f $ has full rank $ k $ on $ O:=\varphi(U'') $, which is a D-open neighborhood of $ x_0 $.
	\end{proof}
	\begin{lemma}\label{p-1}
		Let $ f:\mathcal{L}\times\mathcal{N}\rightarrow \mathcal{M} $ be a smooth map between diffeological \'{e}tale manifolds.
		For each $ x\in\mathcal{L} $, define $ f^{x}:\mathcal{N}\rightarrow \mathcal{M}  $ by $ f^{x}(y)=f(x,y) $.
		Then 
		\begin{center}
			$ \mathsf{rank}~~d(\Pr_1,f)_{(x,y)}=\mathsf{dim}(\mathcal{L})+\mathsf{rank}~~d\big{(}f^{x}\big{)}_{y} $
		\end{center}
		for every $ (x,y)\in \mathcal{L}\times\mathcal{N}  $, 
		where $ \Pr_1:\mathcal{L}\times\mathcal{N}\rightarrow \mathcal{L} $ is the projection on the first factor.
	\end{lemma}
	
	\section{Proof of the stability theorem for diffeological \'{e}tale manifolds}
	
	We first need some preliminary results.
	
	\begin{lemma}\label{lem-2}  
		Let $ \mathcal{M}$ and $\mathcal{N} $   be diffeological \'{e}tale manifolds, and $\mathcal{L} $ be a usual manifold. Suppose that $ h:\mathcal{L}\rightarrow\mathrm{C}^{\infty}(\mathcal{N},\mathcal{M}) $ be a smooth map. If $ (x_0,y_0)\in\mathcal{L}\times\mathcal{N} $ and the map $ h(x_0):\mathcal{N}\rightarrow \mathcal{M}  $  has full rank $ k $ at $ y_0 $, then 
		there exist a D-open neighborhood $ O_{x_0}\subseteq\mathcal{L} $ of $ x_0 $ and a D-open neighborhood $ U_{y_0}\subseteq\mathcal{N} $ of $ y_0 $ such that for every $ x\in O_{x_0}$,
		$ h(x)|_{U_{y_0}}:U_{y_0}\rightarrow \mathcal{M}  $  has  full rank $ k $ on $ U_{y_0} $.
	\end{lemma}
	\begin{proof}  
		Define
		$ \hat{h}:\mathcal{L}\times\mathcal{N}\rightarrow \mathcal{M}$ by $ \hat{h}(x,y)= h(x)(y) $.
		%	Obviously, $ \hat{h}^x=h(x) $ for all $ x\in\mathcal{M} $.
		By Lemma \ref{p-1}, we have
		\begin{center}
			$ \mathsf{rank}~~d(\mathrm{Pr}_1,\hat{h})_{(x_0,y_0)}=\mathsf{dim}(\mathcal{L})+\mathsf{rank}~~d\big{(}h(x_0)\big{)}_{y_0}=\mathsf{dim}(\mathcal{L})+k, $
		\end{center} 
		where $ \Pr_1:\mathcal{L}\times\mathcal{N}\rightarrow \mathcal{L} $ is the projection on the first factor.
		This means that $ (\mathrm{Pr}_1,\hat{h}):\mathcal{L}\times\mathcal{N}\rightarrow \mathcal{L}\times\mathcal{M} $ has full rank $ \mathsf{dim}(\mathcal{L})+k $ at $ (x_0,y_0) $.
		By Proposition \ref{cor-1},	there exist a D-open neighborhood $ O_{x_0}\subseteq\mathcal{L} $ of $ x_0 $ and a 
		D-open neighborhood $ U_{y_0}\subseteq\mathcal{N} $ of $ y_0 $ such that
		$ (\mathrm{Pr}_1,\hat{h}) $ has full rank $ \mathsf{dim}(\mathcal{L})+k $ on $ O_{x_0}\times U_{y_0} $.
		Again by Lemma \ref{p-1}, we get
		\begin{center}
			$ \mathsf{dim}(\mathcal{L})+k=\mathsf{rank}~~d(\mathrm{Pr}_1,\hat{h})_{(x,y)}=\mathsf{dim}(\mathcal{L})+\mathsf{rank}~~d\big{(}h(x)\big{)}_{y} $
		\end{center} 
		or
		\begin{center}
			$ \mathsf{rank}~~d\big{(}h(x)\big{)}_{y}=k, $
		\end{center} 
		for all $ (x,y)\in O_{x_0}\times U_{y_0} $. Hence for every $ x\in O_{x_0}$,
		$ h(x)|_{U_{y_0}}:U_{y_0}\rightarrow \mathcal{M}  $  has full  rank $ k $ on $ U_{y_0} $.
	\end{proof}
	
	\begin{proposition}\label{lem-3}  
		Suppose that $ \mathcal{L} $ is a usual manifold,
		and $\mathcal{M}, \mathcal{K} $ are diffeological \'{e}tale manifolds such that $ \mathcal{K} $ is D-compact. Let $ h:\mathcal{L}\rightarrow\mathrm{C}^{\infty}(\mathcal{K},\mathcal{M}) $ be a smooth map. If $ x_0\in\mathcal{L} $ and the map $ h(x_0):\mathcal{K}\rightarrow \mathcal{M}  $  has full rank $ k $ on $ \mathcal{K} $, then there exists a D-open neighborhood $ O\subseteq\mathcal{L} $ of $ x_0 $ such that for every $ x\in O $,
		$ h(x):\mathcal{K}\rightarrow \mathcal{M}  $  has full rank $ k $ on $ \mathcal{K} $.
	\end{proposition}
	\begin{proof}  
		By Lemma \ref{lem-2}, for each $ y\in \mathcal{K} $, there exist a D-open neighborhood $ O_{x_0,y}\subseteq\mathcal{L} $ of $ x_0 $ and a D-open neighborhood $ U_{y}\subseteq\mathcal{K} $ of $ y $ such that for every $ x\in O_{x_0,y}$,
		$ h(x)|_{U_{y}}:U_{y}\rightarrow \mathcal{M}  $  has full rank $ k $ on $ U_{y} $.
		Since the collection $ \{U_y\}_{y\in\mathcal{K}} $ covers $ \mathcal{K} $ and $ \mathcal{K}  $ is D-compact, we conclude that there are finitely many points $ y_1,\dots,y_n $ in $ \mathcal{K}  $ for which $ \bigcup_{i=1}^nU_{y_i}=\mathcal{K} $.
		Set $ O=\bigcap_{i=1}^n O_{x_0,y_i}$, which is a D-open neighborhood of $ x_0 $ in $ \mathcal{L} $.
		Therefore, 
		$ h(x):\mathcal{K}\rightarrow \mathcal{M}  $  has full rank $ k $ on $ \mathcal{K} $, for every $ x\in O $.
	\end{proof} 
	
	\begin{myproof}.  
		
		\ref{thm-main} \textbf{(a)}.
		
		In light of Theorem \ref{the-2}, parts $ \textbf{(a1)},\textbf{(a2)} $ and $ \textbf{(a3)} $ are obtained as especial cases of Proposition \ref{lem-3}, where $ \mathcal{L}=\mathbb{R} $, $ x_0=0 $, and $ k$ is taken equal to $\mathsf{dim}(\mathcal{M}) $, $ \mathsf{dim}(\mathcal{K}) $, and $ \mathsf{dim}(\mathcal{K})=\mathsf{dim}(\mathcal{M}) $, respectively.
		
		$ \textbf{(a4)} $
		Let $ h:\mathbb{R}\rightarrow \mathrm{C}^{\infty}(\mathcal{K},\mathcal{M}) $ be a smooth homotopy such that $ h(0) $ is a 		local subduction. That is, $ h(0) $ is a surjective diffeological submersion.
		By part $ \textbf{(a1)} $, one can find a positive number $ \epsilon >0 $ such that  
		$ h(t):\mathcal{K}\rightarrow\mathcal{M} $ is a diffeological submersion, for all $ t\in (-\epsilon,\epsilon) $.
		In particular, each $ h(t) $ is a D-open map.
		By \cite[Lemma 4.50(a)]{Lee-TM}, each $ h(t) $ is a D-closed map, too.
		Due to the fact that a diffeological space is locally connected by \cite{PIZ0}, its connected components are both D-open and D-closed.
		Thus we conclude that, if $ C\subseteq \mathcal{K} $ is a connected component, so is $ h(t)(C)\subseteq \mathcal{M} $, for all $ t\in (-\epsilon,\epsilon) $.
		Now fix $ t\in (-\epsilon,\epsilon) $ and take any $ y\in \mathcal{M} $. 
		By surjectivity of $ h(0) $, one  can find a point $ x\in\mathcal{K} $ with $ h(0)(x)=y $.
		But by Cartesian closedness, the map $ \gamma_x:\mathbb{R}\rightarrow\mathcal{M}$ defined by $ \gamma_x(s)=h(st)(x) $ is a smooth path connecting $ h(0)(x)=y $ and $ h(t)(x) $. 
		Thus,
		$ h(0)(x)=y $ and $ h(t)(x) $ belong to the same component. Actually, $ h(t) $ maps the component of $ x $ onto the component of $ y $.
		Therefore, $ h(t) $ is surjective and so a local subduction,
		for all $ t\in (-\epsilon,\epsilon) $.
	\end{myproof}
	\\
	\\
	\begin{myproof}.  
		
		\ref{thm-main} \textbf{(b)}.
		
		\textbf{(b1)}
		Let $ h:\mathbb{R}\rightarrow \mathrm{C}^{\infty}(\mathcal{K},\mathcal{M}) $ be a smooth homotopy such that $ h(0) $ is an injective diffeological immersion. 
		By part $ \textbf{(a2)} $, we can find a positive number $ \epsilon >0 $ such that  
		$ h(t):\mathcal{K}\rightarrow \mathcal{M} $ is a diffeological immersion, for all $ t\in (-\epsilon,\epsilon) $.
		Thus, we only need to show that there exists a positive number $ 0<\delta<\epsilon $ such that $ h(t) $ is injective, for all $ t\in (-\delta,\delta) $.
		Otherwise, for every positive integer $ n $, there are $ t_n\in\mathbb{R} $ and distinct points $ x_n,y_n\in\mathcal{K} $ such that $ |t_n|<\dfrac{1}{n} $ and $ h(t_n)(x_n)=h(t_n)(y_n) $.
		Since $ \mathcal{K} $ is compact, up to suitable subsequences, we get
		$ \displaystyle\lim_{n\rightarrow\infty} x_n= x_0 $ and $ \displaystyle\lim_{n\rightarrow\infty} y_n= y_0 $,
		for some $ x_0,y_0\in\mathcal{K} $. Then
		\begin{center}
			$ h(0)(x_0)=\displaystyle\lim_{n\rightarrow\infty} h(t_n)(x_n)=\displaystyle\lim_{n\rightarrow\infty} h(t_n)(y_n)= h(0)(y_0). $
		\end{center}
		Injectivity of $ h(0) $ implies that $ x_0=y_0 $.
		On the other hand, by Lemma \ref{p-1} and Theorem \ref{the-2}, we observe that
		\begin{center}
			$ (\Pr_1,\hat{h}):(-\epsilon,\epsilon)\times\mathcal{K}\rightarrow (-\epsilon,\epsilon)\times \mathcal{M},\quad (t,x)\mapsto \big{(}t,h(t)(x)\big{)} $
		\end{center}
		is a diffeological immersion. Because $ \mathcal{M} $ is a usual manifold, by Proposition \ref{p-2}, $ (\Pr_1,\hat{h}) $ is actually an immersion of diffeological spaces and therefore, locally injective.
		In particular, there exists a neighborhood $ (-\epsilon',\epsilon')\times O\subseteq (-\epsilon,\epsilon)\times\mathcal{K} $ of $ (0,x_0) $ such that
		$ (\Pr_1,\hat{h})|_{(-\epsilon',\epsilon')\times O} $
		is injective. Due to convergence, for some sufficiently large $ m $, we have 
		$ (t_m,x_m),(t_m,y_m)\in (-\epsilon',\epsilon')\times O $.
		Now from injectivity $ (\Pr_1,\hat{h})|_{(-\epsilon',\epsilon')\times O} $ and the equality
		$ \big{(}t_m,h(t_m)(x_m)\big{)}=\big{(}t_m,h(t_m)(y_m)\big{)} $
		we get
		$ x_m=y_m $ for some $ m $, which is a contradiction with our hypothesis that $ x_m$ and $y_m $ are distinct points.
		
		\textbf{(b2)}
		Let $ h:\mathbb{R}\rightarrow \mathrm{C}^{\infty}(\mathcal{K},\mathcal{M}) $ be a smooth homotopy such that $ h(0) $ is a diffeological embedding. So $ h(0) $ is an injective diffeological immersion. 
		By part $ \textbf{(b1)} $, there exists a positive number $ \epsilon >0 $ such that  
		$ h(t):\mathcal{K}\rightarrow \mathcal{M} $ is an injective diffeological immersion, for all $ t\in (-\epsilon,\epsilon) $.
		Now, in view of \cite[Lemma 4.50(c)]{Lee-TM}, the result is obtained.
		
		\textbf{(b3)}
		It is sufficient to consider the class of diffeomorphisms from $ \mathcal{K} $ onto $ \mathcal{M} $ as the intersection of
		the class of local subductions from $ \mathcal{K} $ onto $ \mathcal{M} $ with
		the class of injective diffeological immersions from $ \mathcal{K} $ to $ \mathcal{M} $.
	\end{myproof}
	\\
	\\
	\begin{myproof}.  
		
		\ref{thm-main} \textbf{(c)}.

		$ \textbf{(c1)} $ 
		By \cite[Lemma 4.50(c)]{Lee-TM}, all elements of $ \mathsf{Inj}^{\infty}(\mathcal{K},\mathcal{M}) $ are D-embeddings.
		So the class of diffeological embeddings from $ \mathcal{K} $ to $ \mathcal{M} $ is equal to the intersection of $ \mathsf{Inj}^{\infty}(\mathcal{K},\mathcal{M}) $ with the class of diffeological immersions from $ \mathcal{K} $ to $ \mathcal{M} $.
		Because the subspace topology is coarser than the D-topology of the subspace diffeology, the result is achieved by $ \textbf{(a2)} $.

		$ \textbf{(c2)} $
		Likewise, the class of diffeomorphisms from $ \mathcal{K} $ onto $ \mathcal{M} $ is equal to the intersection of $ \mathsf{Inj}^{\infty}(\mathcal{K},\mathcal{M}) $ with 
		the class of local subductions from $ \mathcal{K} $ onto $ \mathcal{M} $, which is a D-open subset of $ \mathsf{Inj}^{\infty}(\mathcal{K},\mathcal{M}) $ by $ \textbf{(a4)} $. 
	\end{myproof}

	\section*{Acknowledgement} The author would like to thank Jean-Pierre Magnot and Jordan Watts, the organizers of Special Session on Recent Advances in Diffeology and their Applications, AMS-EMS-SMF  Joint International Meeting, Grenoble, France, 2022. He also expresses his thanks to Patrick Iglesias-Zemmour,  Katsuhiko Kuribayashi, and Enxin Wu.
	\bibliographystyle{amsalpha}

\end{document}